\documentclass{amsart}
\usepackage{graphicx}
 \usepackage{stmaryrd}
\usepackage{color}
\def\({\left (}
\def\){\right )}
\def\<{\left\langle}
\def\>{\right\rangle}

 \newtheorem{thm}{Theorem}[section]
\newtheorem{cor}[thm]{Corollary}
\newtheorem{lem}[thm]{Lemma}

\newtheorem{defn}[thm]{Definition}
\theoremstyle{remark}
\newtheorem{rem}[thm]{Remark}
\newtheorem*{acknowledgement}{Acknowledgement}

\newcommand{\norm}[1]{\left\Vert#1\right\Vert}
\newcommand{\abs}[1]{\left\vert#1\right\vert}
\newcommand{\set}[1]{\left\{#1\right\}}
\newcommand{\Real}{\mathbb R}

\newcommand{\pfrac}[2]{\frac{\partial #1}{\partial #2}}

\begin{document}

\title{Higher order neck analysis of harmonic maps and its applications}

\author{Hao Yin}

\address{Hao Yin,  School of Mathematical Sciences,
University of Science and Technology of China, Hefei, China}
\email{haoyin@ustc.edu.cn }

\begin{abstract}
	In this paper, we prove some refined estimate in the neck region when a sequence of harmonic maps from surfaces blow up. The new estimate allows us to see the shape of the center of the neck region. As an application, we prove an inequality about the nullity and index when blow-up occurs.
\end{abstract}
\maketitle

\section{Introduction}
Suppose that $(M,g)$ and $(N,h)$ are two closed Riemannian manifolds. For a map $u$ from $M$ to $N$, the Dirichlet energy is defined to be
\begin{equation*}
	E(u)=\frac{1}{2}\int_M \abs{\nabla u}^2 dV_g.
\end{equation*}
The critical points of $E$ are called harmonic maps. When the dimension of $M$ is $2$, this functional $E$ is conformally invariant and the theory of harmonic maps becomes more interesting. Among other things, there are the well known energy concentration and the bubbling phenomenon for a sequence of harmonic maps with uniformly bounded energy (\cite{sacks1981existence}).  

For simplicity, we restrict ourselves to the case that $M=B_1$ is the unit ball in $\Real^2$ with flat metric and assume that there is one and only one bubble. More precisely, let $u_i$ be a sequence of harmonic maps from $B_1$ to $N$ (isometrically embedded in $\Real^p$) with uniformly bounded energy, satisfying

(i) $u_i$ converges locally smoothly in $\bar{B}_1\setminus \set{0}$ to a limit map $u_\infty$, which by the removable singularity theorem (\cite{sacks1981existence}), can be extended to be a smooth harmonic map from $B_1$ to $N$;

(ii) there is a sequence $\lambda_i$ going to zero such that $u_i(\lambda_i y)$ converges locally smoothly to a bubble map $\omega:\Real^2 \to N$;

(iii) $\omega$ is the only bubble.

The classical energy identity (see \cite{Jost1991dimensional}, \cite{ding1995energy}, \cite{Qing1995singularities}, \cite{Wang1996Bubble})  and the no-neck theorem (see \cite{parker1993pseudo}, \cite{LW1998}, \cite{QT1997}) imply that
\begin{equation*}
	\lim_{\delta\to 0} \lim_{i\to\infty} \int_{B_\delta\setminus B_{\lambda_i /\delta}} \abs{\nabla u_i}^2 dx =0
\end{equation*}
and
\begin{equation*}
	\lim_{\delta\to 0} \lim_{i\to\infty} \text{osc}_{B_\delta\setminus B_{\lambda_i /\delta}} u_i =0.
\end{equation*}
In fact, the proof gives more. Precisely, there exists some $\alpha>0$ such that (in polar coordinates)
\begin{equation*}
	\abs{r\partial_r u_i}, \abs{\partial_\theta u_i} \leq C(\delta) \eta^\alpha, \qquad \forall r\in (\lambda_i/\delta ,\delta),
\end{equation*}
where $\eta(r,\theta)=r+ \frac{\lambda_i}{r}$.
As a consequence, there exists $P_i\in \Real^p$ such that
\begin{equation}\label{eqn:holder}
	\abs{u_i-P_i} \leq C\eta^\alpha \qquad \text{on} \quad B_\delta \setminus B_{\lambda_i /\delta}.
\end{equation}
It makes sense to call \eqref{eqn:holder} some uniform H\"older estimate of $u_i$ on the neck. 

This paper studies the question: what more can we say about $u_i$ on the neck region? Our main result is in some sense the $C^{1,\alpha}$ uniform estimate on the neck.

\begin{thm}
	\label{thm:main1} Suppose $u_i$ is a sequence of harmonic maps satisfying (i-iii). There exist {\it uniformly bounded} coefficients $p_i, q_i, a_{i}, b_{i}, c_{i}$ and $d_{i}$ such that for any $\alpha\in (0,1)$, $t\in [\log \lambda_i/\delta, \log \delta]$,
	\begin{equation*}
		u_i= p_i + q_i t + a_{i} e^t\cos \theta + b_{i} e^t\sin \theta + c_{i} \lambda_i e^{-t}\cos \theta + d_{i} \lambda_i e^{-t}\sin \theta + O(\eta^{1+\alpha}),
	\end{equation*}
	where $O(\eta^{1+\alpha})$ is some function uniformly bounded by $C \eta^{1+\alpha}$.

	Moreover, 
	\begin{equation}\label{eqn:moreover}
		\abs{\abs{q_i}^2 -2\lambda_i (a_{i}\cdot c_{i} + b_{i}\cdot d_{i})} \leq C \lambda_i^{\alpha/2+1}.
	\end{equation}
\end{thm}

The coefficients above are meaningful in an asymptotic sense. First, it is obvious that $p_i$ converges to $p_\infty$, where the limit map and the bubble touch (by the no neck theorem). Second, the limit map $u_\infty$ has Taylor expansion at the origin
	\begin{equation}\label{eqn:expu}
		u_\infty= p_\infty + a_\infty x_1 + b_\infty x_2 + o(\abs{x}).
	\end{equation}
	Theorem \ref{thm:main1} implies that $\lim_{i\to\infty}a_i=a_\infty$ and $\lim_{i\to\infty} b_i=b_\infty$.  
	Similarly, let $\tilde{\omega}(y_1,y_2)= \omega(\frac{y_1}{\abs{y}^2}, \frac{y_2}{\abs{y}^2})$. By the removable singularity theorem, $\tilde{\omega}$ is smooth at the origin so that we have 
	\begin{equation}\label{eqn:expomega}
		\tilde{\omega}= p_\infty + c_\infty y_1 + d_\infty y_2 + o(\abs{y}).
	\end{equation}
	Theorem \ref{thm:main1} also implies that $\lim_{i\to\infty} c_i=c_\infty$ and $\lim_{i\to\infty} d_i=d_\infty$.
	
	\begin{rem}
		We do not have a precise interpretation for $q_i$ except \eqref{eqn:moreover}, which is not enough to determine $q_i$.	
	\end{rem}

The proof of Theorem \ref{thm:main1} is some neck version of elliptic regularity. The basic idea is to compare the solution of PDE with polynomials (see \cite{caffarelli}, \cite{CC}). Here we adapt an approach in \cite{yin2018schauder}. There is a key lemma in the argument (see Lemma \ref{lem:key}) that improves the regularity by solving Poisson equations. An important observation here is that this result works on cylinders of unknown length uniformly.

As an application of Theorem \ref{thm:main1}, we discuss the geometry of the center part of the neck. Previously, we knew that on each segment of finite length, $u_i$ is close to a constant map uniformly, especially around the center of the neck $\set{t=\frac{1}{2}\log \lambda_i}$. Now, with the help of Theorem \ref{thm:main1}, we can do a scale up to see more details of $u_i$ in this part.
\begin{cor}\label{cor:harmonic}
	Let $u_i$ be as in Theorem \ref{thm:main1}. Then the sequence
	\begin{equation*}
		v_i(s,\theta):=\frac{1}{\sqrt{\lambda_i}} \left( u_i(s+\frac{1}{2}\log \lambda_i, \theta)- (p_i+q_i\frac{1}{2}\log \lambda_i) \right)
	\end{equation*}
	converges smoothly locally on $(-\infty,+\infty)\times S^1$ to
	\begin{equation}\label{eqn:vinf}
		v_\infty(s,\theta)= q_\infty s + a_\infty e^s \cos\theta + b_\infty e^s \sin\theta + c_\infty e^{-s}\cos \theta + d_\infty e^{-s}\sin \theta,
	\end{equation}
	where
	\begin{equation*}
		\abs{q_\infty}^2= 4 (a_\infty\cdot c_\infty + b_\infty \cdot d_\infty).
	\end{equation*}
\end{cor}

In case $(\abs{a_\infty}^2+ \abs{b_\infty}^2)(\abs{c_\infty}^2 + \abs{d_\infty}^2)\ne 0$, the map $v_\infty$ is the result of a balance of influences from both sides of the neck. 
We may derive more consequences when $u_i$'s are weakly conformal, in addition to being harmonic map. Recall that weakly conformal harmonic maps are minimal immersions and that the limit of weakly conformal harmonic maps is a weakly conformal harmonic map. Hence, the limit map $u_\infty$ and the bubble map $\omega$ are weakly conformal. In particular,
\begin{gather*}
	\abs{a_\infty}^2=\abs{b_\infty}^2;\qquad  a_\infty\cdot b_\infty=0; \\ 
	\abs{c_\infty}^2=\abs{d_\infty}^2;\qquad  c_\infty\cdot d_\infty=0. \\ 
\end{gather*}

The conformality of $v_\infty$ puts some restrictions to the coefficients in \eqref{eqn:vinf}.
\begin{thm} \label{thm:main2}
	Suppose	that $u_i$ is a sequence of weakly conformal harmonic maps satisfying (i)-(iii) above. 
	For the $v_\infty$ map given in Corollary \ref{cor:harmonic}, we have
	\begin{gather}
		q_\infty\cdot a_\infty = q_\infty \cdot b_\infty = q_\infty\cdot c_\infty =q_\infty\cdot d_\infty =0;
		\label{eqn:q1}\\
		\abs{q_\infty}^2= 2( a_\infty \cdot c_\infty + b_\infty \cdot d_\infty) ;
		\label{eqn:q2} \\
		a_\infty\cdot d_\infty -b_\infty\cdot c_\infty=0.
		\label{eqn:q3}
	\end{gather}
\end{thm}

When $\dim N=3$, the equations \eqref{eqn:q1}, \eqref{eqn:q2} and \eqref{eqn:q3} imply that the tangent planes (if not degenerate) of the limit map and the bubble at the touching point coincide.
\begin{cor}\label{cor:minimal}
	Assume that $\dim N=3$ and $(\abs{a_\infty}^2+ \abs{b_\infty}^2)(\abs{c_\infty}^2 + \abs{d_\infty}^2)\ne 0$. The planes spanned by $(a_\infty,b_\infty)$ and  $(c_\infty,d_\infty)$ are the same (modulo orientation). Moreover, if $q_\infty=0$, the orientations are opposite and if $q_\infty\ne 0$, one must have $a_\infty=\lambda c_\infty$ and $b_\infty=\lambda d_\infty$ for some positive $\lambda$, in which case the surface parametrized by \eqref{eqn:vinf} is the catenoid. 
\end{cor}

Our second application of Theorem \ref{thm:main1} is an index inequality. Given a harmonic map $u$, the nullity $\mbox{Nul}(u)$ is defined to be the maximal number of linearly independent Jacobi fields along $u$ and the index, $\mbox{Ind}(u)$, is the dimension of the maximal subspaces on which the Hessian of $E$ is negatively definite. We refer to Section \ref{sec:pre} for precise definitions. These quantities related to the second variation of the energy functional are important to applications (see \cite{micallef1988minimal}).

In this paper, we define
\begin{equation*}
	\mbox{NI}(u)=\mbox{Nul}(u) + \mbox{Ind}(u).
\end{equation*} 
We are interested in the change of $\mbox{NI}$ when blow-up occurs. For simplicity, we state the theorem under the one bubble assumption as before. 
\begin{thm}
	\label{thm:main3} Suppose that $u_i$ is a sequence of harmonic maps with uniformly bounded energy from a closed Riemannian surface $(M,g)$ to a closed Riemannian manifold $(N,h)$. Assume that $p\in M$ is the only energy concentration point and that in a coordinate neighborhood of $p$, (i)-(iii) holds. Then
	\begin{equation}\label{eqn:ni}
		\limsup_{i\to\infty} \mbox{NI}(u_i) \leq \mbox{NI}(u_\infty) + \mbox{NI}(\omega)
	\end{equation}
	and
	\begin{equation}\label{eqn:n}
		\limsup_{i\to\infty} \mbox{Nul}(u_i) \leq \mbox{Nul}(u_\infty) + \mbox{Nul}(\omega).
	\end{equation}
\end{thm}
Here we regard $\omega$ as a harmonic map from $S^2$ to $N$.

One should not expect equality in \eqref{eqn:ni} and \eqref{eqn:n} in the general case. Even in the smooth limit case (no bubbling), we can find a sequence of harmonic maps with $\mbox{Nul}(u_i)=0$, but the smooth limit map $u$ has $\mbox{Nul}(u)>0$. 

\begin{rem}
While it is trivial to generalize the above statement to the multi-bubble case, we will have to include a contribution of $\mbox{NI}(\omega)$ even for a ghost bubble.	The nullity of a constant map is the same with $\dim N$, which weakens the theorem. We shall address this issue in a forthcoming paper.
\end{rem}

The paper is organized as follows. Section \ref{sec:pre} contains some preliminary facts which we include for the convenience of the readers. In Section \ref{sec:key}, we prove an important lemma which we use in Section \ref{sec:proof} to prove Theorem \ref{thm:main1}. Section \ref{sec:index} is devoted to the proof of Theorem \ref{thm:main3}.

\begin{acknowledgement}
	The author would like to thank Professor Li Yuxiang for his comments. The research work is supported by NSFC 11471300.
\end{acknowledgement}

\section{Preliminaries} \label{sec:pre}
In this section, we collect a few known results and definitions.
\subsection{Known neck analysis}
Let $u_i$ be a sequence of harmonic maps satisfying (i)-(iii) in the introduction. The neck domain is denoted by $\Omega_i:=B_\delta\setminus B_{\lambda_i/\delta}$. We have defined in the introduction 
\begin{equation*}
	\eta= \abs{x}+\frac{\lambda_i}{\abs{x}}.
\end{equation*}
Notice that we have omitted the subscript $i$ for $\eta$. In the cylinder coordinates $(t,\theta)$, where $t=\log r$, the neck domain becomes
\begin{equation*}
	N_i:= [\log \lambda_i/\delta, \log \delta ]\times S^1
\end{equation*}
and the definition of $\eta$ reads
\begin{equation*}
	\eta(t,\theta)= e^t + \lambda_i e^{-t}.
\end{equation*}

The proof of no neck theorem (see \cite{LW1998} and \cite{QT1997}) gives a small $\delta>0$ and $\alpha>0$ such that
\begin{equation}\label{eqn:pre1}
	\abs{\partial_t u_i}, \abs{\partial_\theta u_i} \leq C \eta^\alpha
\end{equation}
for $(t,\theta)\in N_i$.
Combining the above with the $\varepsilon$-regularity theorem, we obtain
\begin{equation}\label{eqn:pre2}
	\abs{\partial_t^{k_1}\partial_\theta^{k_2} u_i}\leq C(k_1,k_2) \eta^\alpha
\end{equation}
on $N_i$ for any $k_1,k_2\in \mathbb N \cup\set{0}$ and $k_1+k_2>0$.

\begin{defn}\label{defn:O}
	A function $w$ on $N_i$ is said to be an $O(\eta^\alpha)$ if for any $k_1,k_2\in \mathbb N\cup \set{0}$, 
	\begin{equation*}
		\abs{\partial_t^{k_1} \partial_{\theta}^{k_2}w} \leq C(k_1,k_2) \eta^\alpha \qquad \text{on}\quad N_i.
	\end{equation*}
\end{defn}

Due to \eqref{eqn:pre1}, by setting $p_i=u(\frac{1}{2}\log \lambda_i, 0)$, we have
\begin{equation}\label{eqn:start}
	u_i=p_i+ O(\eta^\alpha).
\end{equation}

\subsection{Second variation of harmonic maps}\label{sub:equation}
Suppose that $(M,g)$ and $(N,h)$ are two closed Riemannian manifold and $u:M\to N$ is a harmonic map. A section $v$ of the pullback bundle $u^*TN$ is said to be a vector field along $u$. If $v_1$ and $v_2$ are two vector fields along $u$ and $\phi(s,t)$ is a 2-parameter variation of $u$ satisfying
\begin{equation*}
	v_1=\pfrac{\phi}{s}|_{(s,t)=(0,0)}, \qquad v_2=\pfrac{\phi}{t}|_{(s,t)=(0,0)},
\end{equation*}
then the $E$-index form, or the Hessian of $u$ is defined by
\begin{equation*}
	\mathcal H_u(v_1,v_2)= \frac{\partial^2 E(\phi)}{\partial s \partial t}|_{(s,t)=(0,0)}. 
\end{equation*}
When $\dim M=2$, the energy $E$ is conformal invariant and so is $\mathcal H_u$. Computation shows that
\begin{equation*}
	\mathcal H_u(v_1,v_2)=\int_M \langle J_u(v_1), v_2\rangle dV_g,
\end{equation*}
where
\begin{equation}\label{eqn:J}
	-J_u(v)=\mbox{Tr}_g \nabla^2_u v + \mbox{Tr}_g R(du,v)du.
\end{equation}
Here $R$ is the Riemann curvature tensor of $N$; $\nabla_u$ is the induced connection of the bundle $u^*TN$ and $\mbox{Tr}_g$ is an operator that takes the trace with respect to $g$. The operator $J_u$ is elliptic and self-adjoint so that we can talk about the eigenvalues and eigenfunctions. The index is defined to be the number of negative eigenvalues (counting multiplicity) and the nullity is defined to be the dimension of $\mbox{Ker} J_u$.

For the proof of Theorem \ref{thm:main3}, we shall study the elliptic equation
\begin{equation}\label{eqn:jacobi}
	J_u(v)=\beta v.
\end{equation}
We will derive apriori estimates of $v$ from the above equation, in which the metric $g$ and the map $u$ are considered as known coefficients. In order to be precise, we write down \eqref{eqn:jacobi} in local coordinates.

Recall that the manifold $N$ is embedded in $\Real^p$. Hence, $u$ is a $\Real^p$-valued function, whose components are denoted by $u^\mu$ for $\mu=1,\cdots,p$. We also take local isothermal coordinates $(x^1,x^2)$ in a neighborhood of $M$ so that the metric $g$ is $g_{ab}dx^a dx^b$ where $a,b=1,2$. $v$ as a vector field along $u$ is regarded as a map into $\Real^p$ as well, satisfying the extra restriction
\begin{equation*}
	v(x^1,x^2)\in T_{u}N.
\end{equation*}
Denote the projection map from $\Real^p$ to $T_y N$ by $\Pi(y)$. $\Pi(y)$ is a $p\times 2$ matrix smoothly depending on $y\in N$. We smoothly extend its definition to a neighborhood of $N$ in $\Real^p$ so that it makes sense to say $\partial_\mu \Pi$. It is not hard to see that the following computation is independent of this choice of extension. With the help of $\Pi$, we have
\begin{equation*}
	\nabla_{u;a} v = \Pi(u) (\partial_a v),
\end{equation*}
which we use to compute
\begin{equation}\label{eqn:ju}
	\begin{split}
	J_u (v)&= - g^{ab}\nabla_{u;a} \nabla_{u;b} v - g^{ab}R^N(\partial_a u,v)\partial_b u \\
	&=  - g^{ab}\Pi(u) \partial_a \left( \Pi(u) \partial_b v \right) - g^{ab} R^N(\partial_a u,v)\partial_b u \\ 
	&= -g^{ab}\Pi(u) \partial_{ab}^2 v - g^{ab}\Pi(u) [\partial_\mu \Pi(u) \partial_a u^\mu ] \partial_b v - g^{ab} R^N(\partial_a u,v)\partial_b u. 
	\end{split}
\end{equation}
To see that this is an elliptic linear operator of $v$, we use the fact that $v\in T_u N$, or equivalently,
\begin{equation}\label{eqn:beingnormal}
	\Pi(u) v =v
\end{equation}
for all $(x^1,x^2)$.

We take two derivatives of \eqref{eqn:beingnormal} to get
\begin{equation*}
	(\partial_\mu \Pi)(u) (\partial_a u^\mu) v  + \Pi(u) \partial_a v = \partial_a u
\end{equation*}
and
\begin{equation}
	\begin{split}
	&(\partial^2_{\mu\nu}\Pi(u) \partial_a u^\mu \partial_b u^\nu )v + (\partial_\mu \Pi)(u) (\partial^2_{ab} u^\mu) v   \\
	+&(\partial_\mu \Pi)(u) \partial_a u^\mu \partial_b v + (\partial_\nu \Pi)(u) \partial_b u^\nu \partial_a v + \Pi(u) \partial^2_{ab} v \\
	=& \partial^2_{ab} u.
	\end{split}
	\label{eqn:pi}
\end{equation}
Plugging \eqref{eqn:pi} into \eqref{eqn:ju}, we obtain
\begin{equation}
	\begin{split}
		J_u(v) =& -g^{ab} \partial^2_{ab} v \\
		& - g^{ab}\Pi(u) [\partial_\mu \Pi(u) \partial_a u^\mu ] \partial_b v  +  2 g^{ab}(\partial_\mu \Pi)(u) \partial_a u^\mu \partial_b v  \\
		& + g^{ab}(\partial_\mu \Pi)(u) (\partial^2_{ab} u^\mu) v - g^{ab} R^N(\partial_a u,v)\partial_b u + g^{ab}(\partial^2_{\mu\nu}\Pi(u) \partial_a u^\mu \partial_b u^\nu )v.
	\end{split}
	\label{eqn:Ju}
\end{equation}

We end this section by an observation. Since $N$ is compact, $\Pi$, $R$ and all of their derivatives are bounded. If we work in a coordinate neighborhood $\set{\abs{x}\leq \delta}$ satisfying \\
(1) $\frac{1}{\Lambda}\delta_{ab}\leq g_{ab}\leq \Lambda \delta_{ab}$; \\
 (2)$\abs{\beta}$ and all partial derivatives of $g_{ab}$ and $u$ are bounded by $\Lambda$, \\
then the elliptic estimate gives
\begin{equation*}
	\norm{v}_{C^{1,\alpha}(\set{\abs{x}<\delta/2})}\leq C(\delta,\Lambda,N) \norm{v}_{L^2(\set{\abs{x}<\delta})}
\end{equation*}
for a solution $v$ to the equation
\begin{equation*}
	J_u(v)=\beta v.
\end{equation*}
\section{The key lemma on the cylinder}\label{sec:key}
In this section, we give a proof of the following lemma, which is about the Poisson equation on long cylinders. It plays a key role in the proof of Theorem \ref{thm:main1}.

Recall that $N_i$ and $\eta$ are defined by
\begin{equation*}
	N_i:= [\log \lambda_i/\delta, \log \delta ]\times S^1
\end{equation*}
and
\begin{equation*}
	\eta(t,\theta)= e^{t} + \lambda_i e^{-t}.
\end{equation*}
Denote the Laplace operator on cylinders by $\tilde{\triangle}$.

\begin{lem}\label{lem:key}
	Suppose that $f$ is defined on $N_i$ with $\abs{f}\leq C_1 \eta^\alpha$ for some $0<\alpha\notin \mathbb N$. Then we can find a solution $\tilde{\triangle} v =f$ such that
	\begin{equation*}
		\abs{v}\leq C_2 \eta^{\alpha} \qquad \text{on} \quad N_i.
	\end{equation*}
	Here $C_2$ depends on $C_1$, but not on $i$.
\end{lem}

For a clear presentation, it is easier to prove the lemma on a cylinder centered at $t=0$, namely, $[-L,L]\times S^1$. More precisely, we prove
\begin{lem}\label{lem:centerkey}
	Suppose that $f$ is defined on $[-L,L]\times S^1$ with $\abs{f}\leq C_1 \tilde{\eta}^\alpha$ for some $0<\alpha\notin \mathbb N$, where $\tilde{\eta}(s,\theta)=e^s+e^{-s}$. Then we can find a solution $\tilde{\triangle} v =f$ such that
	\begin{equation*}
		\abs{v}\leq C_2 \tilde{\eta}^{\alpha} \qquad \text{on} \quad [-L,L]\times S^1.
	\end{equation*}
	Here $C_2$ depends on $C_1$, but not on $i$.
\end{lem}
To see these two lemmas are equivalent, we consider a translation $s= t- \frac{1}{2}\log \lambda_i$ and notice that the statement of Lemma \ref{lem:key} remains unchanged with $L=\log\delta - \frac{1}{2}\log \lambda_i$, except that $\eta$ now becomes
\begin{equation*}
	\eta(s,\theta)= \sqrt{\lambda_i} \left( e^s + e^{-s}\right).
\end{equation*}
It is only different from $\tilde{\eta}$ by a constant and we can always multiply both $v$ and $f$ by a constant without changing anything.

\vskip 1cm
Assume without loss of generality that $L$ is an integer. We write
\begin{equation*}
	\mathcal C_i= [i-1,i]\times S^1
\end{equation*}
for $i=-L+1,\cdots,L$.

The following subsections are devoted to the proof of Lemma \ref{lem:centerkey}. In Section \ref{sub:poisson}, we solve the Poisson equation with the right hand side $f\chi_{\mathcal C_i}$, where $\chi_{\mathcal C_i}$ is the characteristic function of $\mathcal C_i$. In Section \ref{sub:harmonic}, we prove a few estimates on the harmonic functions, which we use in Section \ref{sub:proof} to modify the solution obtained in Section \ref{sub:poisson} and sum the modified solutions up to finish the proof.

\subsection{The Poisson equation on cylinder}\label{sub:poisson}

Let the infinite cylinder be denoted by $\mathcal C_\infty$. Suppose that $f$ is a smooth function supported in $\mathcal C_1$ and $\sup_{\mathcal C_\infty} \abs{f}\leq 1$. We are looking for a solution
\begin{equation}\label{eqn:poisson1}
	\tilde{\triangle} u =f
\end{equation}
on $\mathcal C_\infty$ with good estimate of $u$.

Setting $r= e^{s}$ and $x=r\cos \theta$, $y=r\sin \theta$, \eqref{eqn:poisson1} becomes 
\begin{equation*}
	\triangle u = r^{-2} f := \bar{f}.
\end{equation*}
Here $\triangle= \frac{\partial^2}{\partial x^2} + \frac{\partial^2}{\partial y^2}$ and $\bar{f}$ is supported in $B_e \setminus B_1$ and uniformly bounded by $1$. We then have a solution given by
\begin{equation*}
	u(x,y)= - \frac{1}{2\pi}\int_{\Real^2} \bar{f}(x_1,y_1) \log \sqrt{(x-x_1)^2 + (y-y_1)^2} dx_1dy_1.
\end{equation*}
Since $\bar{f}$ is compactly supported and uniformly bounded, it is not hard to show that
\begin{equation*}
	\sup_{\abs{x^2+y^2}\leq e } \abs{u}\leq C \qquad \sup_{\abs{x^2+y^2}\geq e} \abs{u}\leq C \log (x^2+y^2)
\end{equation*}
for some universal constant $C>0$.  
In terms of the cylinder coordinates $(s,\theta)$, we have
\begin{equation}
\label{eqn:cylinderpoisson}
\begin{split}
\abs{u(s,\theta)}&\leq C \qquad \text{for} \quad s\leq 1 \\
\abs{u(s,\theta)}&\leq Cs \qquad \text{for} \quad s> 1.
\end{split}
\end{equation}

For each $i=-L+1, \cdots,L$, let $\chi_{\mathcal C_i}$ be the characteristic function of $\mathcal C_i$. Set $f_i= f \chi_{\mathcal C_i}$ and let $v_i$ be the solution (given above) to the equation
\begin{equation}\label{eqn:vi}
	\tilde{\triangle} v_{i} = f_i.
\end{equation}
Since $\abs{f}\leq C \tilde{\eta}^\alpha$, we have
\begin{equation*}
	\sup_{\mathcal C_\infty} \abs{f_i} \leq C e^{\alpha \abs{i}}.
\end{equation*}
Notice that $f_i$ is supported in $\mathcal C_i$ instead of $\mathcal C_1$, hence after a translation in $s$-direction, \eqref{eqn:cylinderpoisson} implies
\begin{equation}
\label{eqn:solutionvi}
\begin{split}
	\abs{v_i}&\leq Ce^{\alpha\abs{i}} \qquad \text{for} \quad \abs{s}\leq \abs{i} \\
	\abs{v_i}&\leq Ce^{\alpha\abs{i}}(\abs{s}-\abs{i}+1) \qquad \text{for} \quad \abs{s}> \abs{i}. 
\end{split}
\end{equation}

\subsection{Harmonic function on cylinder}\label{sub:harmonic}
We are interested in a bounded harmonic function $u$ on a part of cylinder $[-M,M]\times S^1$. It is well known that we have an expansion
\begin{equation}\label{eqn:wellknown}
	u(s,\theta) = a_0 + b_0 s + \sum_{n=1}^\infty \left( a_n e^{ns}\cos n\theta + b_n e^{ns} \sin n\theta + c_n e^{-ns}\cos n\theta + d_n e^{-ns} \sin n\theta \right).
\end{equation}
The next lemma gives estimates on these coefficients in terms of the $C^0$ norm of $u$.
\begin{lem}\label{lem:harmonic1}
	Suppose that $M\geq 1$ and $u$ is a harmonic function on $[-M,M]\times S^1$. If $\norm{u}_{C^0([-M,M]\times S^1)}\leq \epsilon$, then
	\begin{equation*}
		\abs{a_0}<\epsilon, \qquad \abs{b_0} \leq 2\epsilon /M 
	\end{equation*}
	and
	\begin{equation*}
		\abs{a_n}, \abs{b_n}, \abs{c_n}, \abs{d_n} \leq 4 \epsilon e^{-nM}.
	\end{equation*}
\end{lem}
\begin{proof}
	Consider $w(s)= \frac{1}{2\pi}\int_{S^1} u(s,\theta)d\theta$. Then $w(s)=a_0+b_0 s$. By our assumption on the $C^0$ norm of $u$, we have
	\begin{equation*}
		\abs{a_0+b_0 s} \leq \epsilon
	\end{equation*}
	for any $s\in [-M,M]$, from which the estimates for $a_0$ and $b_0$ follow. Similarly, setting $w(s)= \frac{1}{\pi} \int_{S^1} u(s,\theta) \cos n\theta d\theta$, we have
	\begin{equation*}
		w(s)= a_n e^{ns} + c_n e^{-ns}.
	\end{equation*}
	Again, the assumption on the $C^0$ norm of $u$ implies that $\abs{w(s)}\leq 2\epsilon$ for $s\in [-M,M]$.
	In particular, we have
	\begin{equation*}
		\left\{
			\begin{array}[]{l}
				\abs{a_n e^{nM} + c_n e^{-nM}}\leq 2\epsilon\\
				\abs{a_n e^{-nM} + c_n e^{nM}}\leq 2\epsilon,
			\end{array}
		\right.
	\end{equation*}
	which implies that
	\begin{equation*}
		\left\{
			\begin{array}[]{l}
				\abs{a_n e^{2nM} + c_n}\leq 2\epsilon e^{nM}\\
				\abs{a_n e^{-2nM} + c_n}\leq 2\epsilon e^{-nM}.
			\end{array}
		\right.
	\end{equation*}
	Subtracting the two inequalities above, we have
	\begin{equation*}
		\abs{a_n} \leq \frac{2\epsilon \left( e^{nM} + e^{-nM} \right)}{e^{2nM}-e^{-2n M}}.
	\end{equation*}
	Our estimate for $a_n$ follows by noticing that $M\geq 1$ and $n\geq 1$. The other estimates are proved similarly.
\end{proof}

Given the expansion \eqref{eqn:wellknown} and any $k\geq 0$, we define 
\begin{equation*}
	P_0= a_0+b_0 s
\end{equation*}
for $k=0$ and
\begin{equation*}
	P_k = a_0 + b_0 s + \sum_{n=1}^k \left( a_n e^{ns}\cos n\theta + b_n e^{ns} \sin n\theta + c_n e^{-ns}\cos n\theta + d_n e^{-ns} \sin n\theta \right)
\end{equation*}
for $k>0$.
A corollary of Lemma \ref{lem:harmonic1} is that there is a constant $C(k)$ depending on $k$, such that 
\begin{equation}\label{eqn:pk}
	\abs{P_k}(s,\theta) \leq C(k)\epsilon 
\end{equation}
for $s\in [-M,M]$.

We are also interested in the remainder
\begin{equation*}
u-P_k = \sum_{n=k+1}^\infty \left( a_n e^{ns}\cos n\theta + b_n e^{ns} \sin n\theta + c_n e^{-ns}\cos n\theta + d_n e^{-ns} \sin n\theta \right).
\end{equation*}

\begin{lem}\label{lem:kplusone}
	Suppose that $M\geq 1$ and $u$ is a harmonic function on $[-M,M]\times S^1$ satisfying $\norm{u}_{C^0([-M,M]\times S^1)}\leq \epsilon$. There exists some constant $C$ depending only on $k$, not $M$, such that
\begin{equation*}
	\abs{u-P_k} \leq \frac{C\epsilon}{e^{(k+1)M}} e^{(k+1) \abs{s}} \qquad \text{for} \quad s\in [-M,M].
\end{equation*}
\end{lem}

\begin{proof}
First, we claim that it suffices to prove	
\begin{equation}\label{eqn:p2k}
	\abs{u-P_{2k}} \leq \frac{C\epsilon}{e^{(k+1)M}} e^{(k+1) \abs{s}}.
\end{equation}
By Lemma \ref{lem:harmonic1}, for any $l=k+1,\cdots,2k$, we have
\begin{equation*}
	\abs{a_l}, \abs{b_l}, \abs{c_l}, \abs{d_l} \leq C\epsilon e^{-lM}.
\end{equation*}
Therefore, for $(s,\theta)\in [-M,M]\times S^1$,
\begin{equation}\label{eqn:forl}
	\abs{a_{l} e^{ls} \cos l \theta + b_le^{ls}\sin l\theta + c_l e^{-ls}\cos l\theta + d_l e^{-ls}\sin l\theta} \leq C\epsilon e^{-lM} e^{l\abs{s}}.
\end{equation}
By summing \eqref{eqn:forl} up for $l=k+1,\cdots,2k$ and noticing that $e^{\abs{s}-M}\leq 1$, we obtain 
\begin{equation*}
	\abs{P_{2k}-P_{k+1}} \leq C k \epsilon e^{-(k+1)M} e^{(k+1)\abs{s}},
\end{equation*}
from which our claim follows.

For the proof of \eqref{eqn:p2k}, we compute
\begin{equation*}
	w(s)=\frac{1}{\pi}\int_{S^1} (u-P_{2k})^2 d\theta = \sum_{n=2k+1}^\infty \left( a_n e^{ns} + c_n e^{-ns} \right)^2 + \left( b_n e^{ns} + d_n e^{-ns} \right)^2.
\end{equation*}
Direct computation shows 
\begin{equation}\label{eqn:convexw}
 w''(s)\geq 4(k+1)^2 w(s). 
\end{equation}
In fact, for each $n>2k$, we have
\begin{eqnarray*}
	\frac{d^2}{ds^2} \left( a_n e^{ns} + c_n e^{-n s} \right)^2 &=& 4n^2 \left(  a_n^2 e^{2ns} + b_n^2 e^{-2n s} \right) \\
	&\geq& 4(k+1)^2\left(  a_n^2 e^{2ns} + b_n^2 e^{-2n s} \right) \times 2 \\
	&\geq& 4(k+1)^2 \left( a_n e^{ns}+b_n e^{-ns} \right)^2.
\end{eqnarray*}

By comparing \eqref{eqn:convexw} with ODE $w''=4(k+1)^2 w$ and noticing the bound $\abs{w}\leq C\epsilon^2$ on $[-M,M]$ (by \eqref{eqn:pk}), we obtain
\begin{equation}\label{eqn:ws}
	w(s)\leq C\epsilon^2 e^{-2(k+1)M} e^{2(k+1)s}.
\end{equation}
Since $u-P_{2k}$ is harmonic, the desired estimate \eqref{eqn:p2k} follows from \eqref{eqn:ws} and the elliptic estimate.
\end{proof}

\subsection{The proof of the key lemma}\label{sub:proof}
Recall that in Section \ref{sub:poisson}, we have defined $v_i$ by solving the Poisson equation \eqref{eqn:vi} on cylinder. If we could sum up these $v_i$'s, we could obtain a solution to \eqref{eqn:poisson1}. However, it seems that the sum is never smaller than $\tilde{\eta}^\alpha$. The idea is to modify $v_i$ by subtracting a harmonic function.

For $i=-1,0,1$, we simply set $u_i=v_i$. In this case, \eqref{eqn:solutionvi} implies that
\begin{equation}\label{eqn:u01}
	\abs{u_i}\leq C(\abs{s}+1) \qquad \text{for} \quad s\in [-L,L].
\end{equation}

For $\abs{i}>1$, by its definition, the function $v_i$ is harmonic on the cylinder $[-\abs{i}+1,\abs{i}-1]\times S^1$ with a uniform bound $C e^{\alpha \abs{i}}$ (see \eqref{eqn:solutionvi}). 
Let $k$ be the integer determined by $k<\alpha<k+1$. Set $u_i= v_i -P_k$, which is estimated as follows.
\begin{itemize}
	\item 
For $\abs{s}\leq \abs{i}-1$, Lemma \ref{lem:kplusone} (with $M=\abs{i}-1\geq 1$) implies that 
\begin{equation}\label{eqn:good1}
	\abs{u_i}(s) \leq C(k)  e^{\alpha \abs{i}} e^{-(k+1)\abs{i}} e^{(k+1)s}.
\end{equation}
\item 
For $L\geq \abs{s}> \abs{i}-1$, the definition of $u_i$ and \eqref{eqn:solutionvi} imply that
\begin{equation}\label{eqn:good2}
	\abs{u_i}(s) \leq  \abs{P_k}(s) + C e^{\alpha \abs{i}} (\abs{s}-\abs{i}+1). 
\end{equation}
\end{itemize}

Now, set
\begin{equation*}
	v=\sum_{i=-L+1}^L u_i.
\end{equation*}
We claim that $v$ is the solution needed in Lemma \ref{lem:centerkey}. Since it is a finite sum, it is easy to see that $v$ solves the Poisson equation \eqref{eqn:poisson1}. It remains to check that $\abs{v}\leq C \tilde{\eta}^\alpha$.

The estimate for $P_k$ in \eqref{eqn:good2} depends on $k$. Hence, we discuss first the case $k\ge 1$.
When $k\geq 1$ and $\abs{s}\geq \abs{i}-1$, Lemma \ref{lem:harmonic1} gives 
\begin{eqnarray*}
	\abs{P_k}(s) &\leq& C e^{\alpha \abs{i}} (1+ \frac{\abs{s}}{\abs{i}}) +\sum_{n=1}^k C e^{\alpha\abs{i}} e^{-n \abs{i}} e^{n \abs{s}}.
\end{eqnarray*}
Using the fact that $e^{\abs{s}-\abs{i}}\geq e^{-1}$ and $(1+\abs{s}/\abs{i})\leq C e^{\abs{s}-\abs{i}}$, we have 
\begin{equation}
	\abs{P_k}(s) \leq C  e^{\alpha \abs{i}} e^{k (\abs{s}-\abs{i})}.
	\label{eqn:goodpk}
\end{equation}
Combining \eqref{eqn:goodpk} with \eqref{eqn:good2}, we get (for $i\ne -1,0,1$, $k\geq 1$ and $\abs{s}>\abs{i}-1$)
\begin{equation}\label{eqn:good3}
	\abs{u_i}(s) \leq C(k) e^{\alpha \abs{i}} e^{k(\abs{s}-\abs{i})}.
\end{equation}

Now we fix $s$ and sum up $u_i$ for $i=- L+1, \cdots, L$, to get
\begin{equation}\label{eqn:sum}
	\begin{split}
	\sum_{i=-L+1}^L \abs{u_i (s,\theta)} &\leq \left( \sum_{i=-1,0,1}+ \sum_{1<\abs{i}< \abs{s}+1} + \sum_{L\geq \abs{i}\geq \abs{s}+1;i\ne -L}\right)  \abs{u_i}(s,\theta)\\
	&\leq C  (\abs{s}+1) + C e^{k\abs{s}} \sum_{\abs{i}< \abs{s}+1} e^{(\alpha-k) \abs{i}} + C  e^{(k+1)s} \sum_{\abs{i}\geq \abs{s}+1} e^{ (\alpha-k-1)\abs{i}} \\
	&\leq C(\alpha,k) e^{\alpha \abs{s}}.
	\end{split}
\end{equation}
Here, we have used \eqref{eqn:u01} for the first sum, \eqref{eqn:good3} for the second one and \eqref{eqn:good1} for the last one. This concludes the proof of Lemma \ref{lem:centerkey} when $k\geq 1$.

When $k=0$, Lemma \ref{lem:harmonic1} implies
\begin{eqnarray*}
	\abs{P_0}(s) &\leq& C e^{\alpha \abs{i}} (1+ \frac{\abs{s}}{\abs{i}}). 
\end{eqnarray*}
Combine this with \eqref{eqn:good2}, we get (for $i\ne -1,0,1$ and $\abs{s}>\abs{i}-1$ and $k=0$)
\begin{equation}\label{eqn:good4}
	\abs{u_i}(s) \leq C  e^{\alpha \abs{i}} \left( 1+ \frac{\abs{s}}{\abs{i}} + \abs{s} - \abs{i} \right)\leq C  e^{\alpha \abs{i}} \left( \abs{s} - \abs{i}+1 \right).
\end{equation}
We can now argue as in \eqref{eqn:sum}, except that we use \eqref{eqn:good4} instead of \eqref{eqn:good3} to estimate the second sum. More precisely, 
\begin{eqnarray*}
	\sum_{1<\abs{i}< \abs{s}+1} \abs{u_i}(s,\theta) &\leq & \sum_{1< \abs{i} <\abs{s}+1} C e^{\alpha \abs{i}} (\abs{s}- \abs{i}+1) \\
	&\leq& C e^{\alpha \abs{s}}. 
\end{eqnarray*}
In fact, the last inequality above is equivalent to
\begin{eqnarray*}
	\sum_{1< \abs{i}< \abs{s}+1} e^{\alpha (\abs{i}-\abs{s})} (\abs{s}- \abs{i}+1)&\leq& C,
\end{eqnarray*}
which, by setting $j= \abs{s}+ 1 -\abs{i}$, follows from the obvious fact
\begin{equation*}
	\sum_{j=0}^\infty e^{-\alpha (j-1)} j \leq C.
\end{equation*}
This finishes the proof of Lemma \ref{lem:centerkey}.

\section{$C^{1,\alpha}$ neck analysis}\label{sec:proof}

In this section, we prove Theorem \ref{thm:main1} and its corollaries. 

\subsection{Proof of Theorem \ref{thm:main1}}\label{sub:main1}
For simplicity, we omit the subscript $i$ and write $u$ for $u_i$, $\lambda$ for $\lambda_i$, etc. 

Recall that there is a vector $p$ such that $u-p$ is $O(\eta^\alpha)$ for some $\alpha>0$ (see \eqref{eqn:start} in Section \ref{sec:pre}). More precisely, there is a constant $C_1>0$ (independent of $i$) such that
\begin{equation*}
	\abs{u-p}(t,\theta)\leq C_1 \eta^\alpha
\end{equation*}
for $t\in [\log \lambda/\delta,\log \delta]$.

\begin{rem}
	From now on, we shall use $C_2, C_3,\cdots$ to denote constants that depends on $C_1$. Hence, they will depend on the sequence of maps $u_i$ and the geometry of the target manifold, but not on $i$. In this paper, we fix $\delta$, so we allow them to depend on $\delta$ as well.
\end{rem}

Since the second fundamental form $A(u)$ is smooth, $A(u)=A(p)+O(\eta^\alpha)$. By Definition \ref{defn:O}, $\tilde{\nabla} u$ is still $O(\eta^\alpha)$. Here $\tilde{\nabla}$ is the gradient operator with respect to the cylinder metric. Hence, the right hand side of the harmonic map equation satisfies
\begin{equation*}
	\abs{A(u)(\tilde{\nabla} u, \tilde{\nabla} u)}\leq C_2 \eta^{2\alpha},
\end{equation*}
for $t\in [\log \lambda/\delta, \log \delta]$, as long as $\alpha<1$.

\begin{rem}
	By Definition \ref{defn:O}, when we say that a function is in $O(\eta^\alpha)$, we mean there exist constants $C(k_1,k_2)$ such that the inequalities in Definition \ref{defn:O} hold. In case necessary, we use an inequality as above to indicate that the constants involved depend on $C_1$.
\end{rem}

Lemma \ref{lem:key} gives a solution $v$ to the equation
\begin{equation}\label{eqn:poissonv}
	\tilde{\triangle} v = A(u)(\tilde{\nabla} u, \tilde{\nabla} u),
\end{equation}
satisfying
\begin{equation*}
	\abs{v}\leq C_3\eta^{2\alpha}, \qquad \forall t\in [\log \lambda/\delta,\log\delta].
\end{equation*}
By the elliptic estimates for equation \eqref{eqn:poissonv}, $v$ is $O(\eta^{2\alpha})$.
Hence, there is a harmonic function $h$ on $N$ such that
\begin{equation*}
	u=h+O(\eta^{2\alpha}).
\end{equation*}
The $C^0$ norm of $h$ depends on the $C^0$ norm of $u$ and the $C^0$ norm of $v$, which by Lemma \ref{lem:key} depends on the constants $C(k_1,k_2)$ in the definition of $u\in O(\eta^\alpha)$. 

Being a harmonic function, $h$ has an expansion discussed in Section \ref{sub:harmonic}. In particular, Lemma \ref{lem:harmonic1} gives estimates for the coefficients in the expansion. Since Lemma \ref{lem:harmonic1} is formulated for the cylinder $[-M,M]\times S^1$, we now translate it into a version that works on $[\log (\lambda/\delta), \log \delta]\times S^1$.
\begin{lem}
	\label{lem:harmonic11}
	Suppose that $h$ is a harmonic function on $[\log(\lambda/\delta),\log \delta]\times S^1$ and $\norm{h}_{C^0([\log(\lambda/\delta),\log \delta]\times S^1)}\leq C$. Then
\begin{eqnarray*}
	h(t,\theta) &=&  a_0 + b_0 (t-\frac{1}{2}\log \lambda)  \\ 
	&& + \sum_{n=1}^\infty \left( a_n e^{nt}\cos n\theta + b_n e^{nt} \sin n\theta + c_n \lambda e^{-nt}\cos n\theta + d_n \lambda e^{-nt} \sin n\theta \right)
\end{eqnarray*}
with
	\begin{equation*}
		\abs{a_0}<\tilde{C}, \qquad \abs{b_0} \leq \tilde{C} / (-\log \lambda) 
	\end{equation*}
	and
	\begin{equation*}
		\abs{a_n}, \abs{b_n}, \abs{c_n}, \abs{d_n} \leq \tilde{C}.
	\end{equation*}
	Here $\tilde{C}$ depends only on $C$ and $\delta$, not on $\lambda$.
\end{lem}
Set
\begin{equation}\label{eqn:pq}
	p=a_0-\frac{1}{2}b_0 \log \lambda \quad \text{and} \quad q=b_0.
\end{equation}
If $2\alpha<1$, Lemma \ref{lem:kplusone} implies that
\begin{equation}\label{eqn:pqt}
	u=p+qt+ O(\eta^{2\alpha}).
\end{equation}

We claim that $q$ is so small that $qt$ is absorbed into $O(\eta^{2\alpha})$ and hence $u=p+O(\eta^{2\alpha})$. To see this, we use the Pohozaev identity of harmonic maps, i.e.
\begin{equation}\label{eqn:poho}
	\int_{S^1} \abs{\partial_t u}^2 d\theta = \int_{S^1} \abs{\partial_\theta u}^2 d\theta.
\end{equation}
By \eqref{eqn:pqt}, the above equation implies that
\begin{equation*}
	\abs{q}^2 + q\cdot O(\eta^{2\alpha}) + O(\eta^{4\alpha})= O(\eta^{4\alpha}).
\end{equation*}
By the Young's inequality, we have
\begin{equation*}
	\abs{q}^2 \leq O(\eta^{4\alpha}).
\end{equation*}
Since $q$ is a constant (independent of $t$), we may evaluate the above inequality at $t=\frac{1}{2}\log \lambda$ (the center of the neck) to see that $\abs{q}\leq C_4 (\sqrt{\lambda})^{2\alpha}$. By the definition of $\eta$,
\begin{equation*}
	\abs{qt}\leq C_5 \eta^{2\alpha}
\end{equation*}
on the neck. This proves the claim, which allows us to repeat the above argument (with $2\alpha$ in place of $\alpha$) until $2\alpha$ becomes larger than $1$. (In case $2\alpha=1$, we take a slightly smaller $\alpha$ and argue as in the case $2\alpha<1$.)

If $2>2\alpha>1$, Lemma 3.4 and Lemma \ref{lem:harmonic11} imply that
\begin{equation}\label{eqn:goodu}
	u=p+qt+ae^t\cos\theta + be^t\sin \theta + c \lambda e ^{-t}\cos\theta + d\lambda e^{-t}\sin\theta + O(\eta^{2\alpha})
\end{equation}
where $a,b,c,d$ are bounded by constants independent of $\lambda$. The uniform bound for $q$ follows from \eqref{eqn:moreover} in Theorem \ref{thm:main1}, which will be proved in a minute. Assuming this, the uniform bound for $p$ follows from \eqref{eqn:pq}.

To conclude the proof of Theorem \ref{thm:main1}, it remains to show \eqref{eqn:moreover}. By recycling the variable $\alpha$, we write $1+\alpha$ ($\alpha\in (0,1)$) in the place of $2\alpha$ in \eqref{eqn:goodu}.

By \eqref{eqn:goodu}, we compute
\begin{equation*}
	\partial_t u=q +  a e^t \cos \theta + b e^t \sin \theta - c \lambda e^{-t} \cos\theta - d\lambda e^{-t} \sin \theta + O(\eta^{1+\alpha})
\end{equation*}
and
\begin{equation*}
	\partial_\theta u= - a e^t \sin \theta + b e^t \cos \theta - c\lambda e^{-t} \sin\theta + d\lambda e^{-t} \cos \theta + O(\eta^{1+\alpha}).
\end{equation*}
Again, the Pohozaev identity \eqref{eqn:poho} implies
\begin{eqnarray*}
	&& 2 \abs{q}^2 + \abs{a e^t- c\lambda e^{-t}}^2 + \abs{b e^t -d\lambda e^{-t}}^2 + q\cdot O(\eta^{1+\alpha}) \\
	&=&  \abs{a e^t + c\lambda e^{-t}}^2 + \abs{b e^t + d\lambda e^{-t}}^2 + O(\eta^{2+\alpha}),
\end{eqnarray*}
which is
\begin{equation*}
	\abs{q}^2 = 2 \lambda (a\cdot c+ b\cdot d) + q\cdot O(\eta^{1+\alpha}) + O(\eta^{2+\alpha}).
\end{equation*}
Since $\lambda(a\cdot c+ b\cdot d)$ is $O(\eta^2)$, we may use the Young's inequality again to see that $\abs{q}\leq C\sqrt{\lambda}$, so that $q\cdot O(\eta^{1+\alpha})$ can be absorbed into $O(\eta^{2+\alpha})$.
Therefore,
\begin{equation*}
	\abs{q}^2 = 2\lambda (a \cdot c +  b\cdot d) + O(\eta^{2+\alpha}).
\end{equation*}
Since $q$ is a constant, we may evaluate the above equation at $t=\frac{1}{2}\log \lambda$ to get \eqref{eqn:moreover}.

\subsection{The shape of the center of the neck}
As a result of \eqref{eqn:holder}, the image of $u_i$ disappears when $i$ goes to infinity and $\delta$ goes to zero. That is exactly why this result is named `no neck theorem'. 
Now with the help of Theorem \ref{thm:main1}, we may consider a scaling of the target manifold, or a scaling of the Euclidean space $\Real^p$ (in which $N$ is embedded isometrically), to see the shape of the center part of the neck. More precisely, for any $M>0$, we define for $(s,\theta)\in [-M,M]\times S^1$,
\begin{equation*}
	v_i(s,t):=\frac{1}{\sqrt{\lambda_i}}\left( u_i(s+\frac{1}{2}\log \lambda_i, \theta) - (p_i+q_i\frac{1}{2}\log \lambda_i) \right).
\end{equation*}

By Definition \ref{defn:O}, if $w$ is $O(\eta^\alpha)$, then $\tilde{w}(s,\theta)= \frac{1}{\sqrt{\lambda_i}}w(s+\frac{1}{2}\log \lambda_i,\theta)$ satisfies
\begin{equation}\label{eqn:ot}
	\abs{\partial_s^{k_1}\partial_\theta^{k_2}\tilde{w}}\leq C(k_1,k_2) (\sqrt{\lambda_i})^{\alpha-1} \left( e^s+e^{-s} \right)^{\alpha}
\end{equation}
for some constants $C(k_1,k_2)$. In what follows, we denote any function satisfying \eqref{eqn:ot} by $\tilde{O}(\alpha-1)$. For fixed $M$ and $\alpha>1$, when $i$ goes to $\infty$ and $\lambda_i$ goes to zero, $\tilde{O}(\alpha-1)$ converges smoothly to zero.

Theorem \ref{thm:main1} implies that on $[-M,M]\times S^1$, we have
\begin{equation*}
	v_i(s,\theta)=\frac{q_i}{\sqrt{\lambda_i}}s+  a_i e^s \cos\theta + b_i e^s \sin \theta + c_i e^{-s}\cos \theta + d_i e^{-s} \sin \theta + \tilde{O}(\alpha-1), 
\end{equation*}
where $q_i/\sqrt{\lambda_i}$, $a_i$, $b_i$, $c_i$ and $d_i$ are uniformly bounded. Hence, by passing to a subsequence if necessary, we obtain a limit $v_\infty$ in $C^\infty$ topology as required in Corollary \ref{cor:harmonic}.

\vskip 8pt
For the proof of Theorem \ref{thm:main2}, we assume 
\vskip 5pt
\noindent (iv) $u_i$ are weakly conformal.
\vskip 5pt
Since $u_i$ are harmonic maps, they are (branched) minimal immersions. Being weakly conformal is a property that remains valid after scaling and passes on to the limit. Hence, the map $v_\infty$ in Corollary \ref{cor:harmonic} is weakly conformal. Using this fact, we give the proof of Theorem \ref{thm:main2} as follows.
\begin{proof}[Proof of Theorem \ref{thm:main2}]
For simplicity, we omit the subscript $\infty$ in the following computation.
Recall that the limit of the center neck region is given by
\begin{equation*}
	v(s,\theta)= qs + a e^s\cos \theta + b e^s \sin \theta +c e^{-s}\cos\theta + de^{-s}\sin \theta. 
\end{equation*}

Direct computation gives
\begin{eqnarray*}
	\partial_s v &=&  q+ a e^s\cos \theta + b e^s \sin \theta -c e^{-s}\cos\theta - de^{-s}\sin \theta \\
	\partial_\theta v &=& - a e^s\sin \theta + b e^s \cos \theta -c e^{-s}\sin\theta + de^{-s}\cos \theta. 
\end{eqnarray*}
Hence, we have
\begin{eqnarray*}
	\abs{\partial_s v}^2 &=& \abs{q}^2 + 2 q\cdot a e^s \cos \theta +2  q\cdot b e^s \sin \theta -2  q\cdot c e^{-s}\cos \theta -2  q\cdot d e^{-s}\sin \theta \\
	&& + \abs{a}^2 e^{2s} \cos^2 \theta +2 a\cdot b e^{2s} \cos \theta \sin\theta -2 a\cdot c \cos^2 \theta -2 a\cdot d \cos\theta \sin\theta \\
	&& + \abs{b}^2 e^{2s} \sin^2 \theta -2 b\cdot c \cos\theta\sin\theta -2 b\cdot d \sin^2\theta  \\
	&& + \abs{c}^2 e^{-2s} \cos^2\theta +2 c\cdot d e^{-2s} \cos\theta\sin\theta + \abs{d}^2 e^{-2s}\sin^2\theta
\end{eqnarray*}
and
\begin{eqnarray*}
	\abs{\partial_\theta v}^2 &=& \abs{a}^2 e^{2s}\sin^2\theta  -2a\cdot b e^{2s} \cos\theta\sin\theta + 2a\cdot c \sin^2\theta - 2a\cdot d \sin\theta\cos\theta \\
	&& + \abs{b}^2 e^{2s}\cos^2 \theta - 2 b\cdot c \cos\theta\sin\theta +  2 b\cdot d \cos^2 \theta  \\
	&& + \abs{c}^2 e^{-2s}\sin^2\theta - 2 c\cdot d e^{-2s}\cos\theta\sin\theta + \abs{d}^2 e^{-2s}\cos^2\theta.
\end{eqnarray*}
Being weakly conformal requires that $\abs{\partial_s v}= \abs{\partial_\theta v}$ for all $(s,\theta)$, which implies that 
\begin{eqnarray}\nonumber
	\abs{q}^2&=& 2(a\cdot c + b\cdot d) \\\nonumber
	q\cdot a &=& q\cdot b = q\cdot c = q\cdot d =0 \\ 
	\label{eqn:a1}	\abs{a}&=& \abs{b} \\
	\label{eqn:a2} \abs{c}&=& \abs{d} \\
	\label{eqn:a3} a\cdot b &=&  c\cdot d =0.
\end{eqnarray}
Bearing this in mind, we compute
\begin{eqnarray*}
	\partial_t v \cdot \partial_\theta v &=&  -\abs{a}^2 e^{2s}\cos\theta\sin \theta + a\cdot b e^{2s}\cos^2\theta - a\cdot c \cos\theta\sin\theta + a\cdot d \cos^2\theta \\
	&& -a\cdot b e^{2s}\sin^2\theta +\abs{b}^2 e^{2s}\cos\theta\sin\theta -b\cdot c \sin^2\theta + b\cdot d \cos\theta\sin\theta \\
	&& + a\cdot c \cos\theta\sin\theta -b\cdot c \cos^2\theta +\abs{c}^2 e^{-2s}\cos\theta\sin\theta -c\cdot d e^{-2s}\cos^2\theta \\
	&& + a\cdot d \sin^2\theta - b\cdot d \cos\theta\sin\theta + c\cdot d e^{-2s}\sin^2\theta - \abs{d}^2e^{-2s}\cos\theta \sin\theta \\
	&=& (a\cdot d - b\cdot c).
\end{eqnarray*}
Hence, the conformality of $v_\infty$ further implies that
\begin{equation*}
	a\cdot d - b\cdot c =0.
\end{equation*}
This concludes the proof of Theorem \ref{thm:main2}.
\end{proof} 

\begin{rem}
\eqref{eqn:a1}, \eqref{eqn:a2} and $\eqref{eqn:a3}$ are nothing but the requirement that $u_\infty$ and $\omega$ are both weakly conformal.
\end{rem}

Now, let's turn to the proof of Corollary \ref{cor:minimal}.

Recall that $u_i$ maps into the manifold $N$, which is a submanifold of $\Real^p$. After scaling and passing to the limit, $v_\infty$ maps into the tangent space of $N$ at $p_\infty$, which we assume is $\Real^n(\subset \Real^p)$. Hence, we may regard $q_\infty, a_\infty,b_\infty,c_\infty$ and $d_\infty$ as vectors in $\Real^n$.

Since $\dim N=3$ and the pairs $(a_\infty,b_\infty)$ and $(c_\infty,d_\infty)$ span two planes by the assumption of Corollary \ref{cor:minimal}, it suffices to exclude the possibility that $a_\infty,b_\infty,c_\infty$ and $d_\infty$ span the whole $\Real^3$. We argue by contradiction. If $\mbox{Span}\set{a_\infty,b_\infty,c_\infty,d_\infty}=\Real^3$, then \eqref{eqn:q1} implies that $q_\infty=0$ and hence \eqref{eqn:q2} and \eqref{eqn:q3} together imply that
\begin{equation}
	\left\{
		\begin{array}[]{l}
			a_\infty\cdot c_\infty + b_\infty\cdot d_\infty =0 \\
			a_\infty \cdot d_\infty -b_\infty \cdot c_\infty =0.
		\end{array}
		\right.
	\label{eqn:position}
\end{equation}
The equation \eqref{eqn:position} describes the relative position of two planes spanned by $(a_\infty,b_\infty)$ and $(c_\infty,d_\infty)$ respectively. The relative position is independent of the choice of orthonormal basis (with orientation) in the plane. Indeed, for any $\theta\in [0,2\pi]$, if we set
\begin{equation*}
	\left\{
		\begin{array}[]{l}
			\tilde{a}_\infty = \cos\theta a_\infty + \sin \theta b_\infty \\
			\tilde{b}_\infty = -\sin\theta a_\infty + \cos \theta b_\infty,
		\end{array}
		\right.
\end{equation*}
then it is straightforward to check that \eqref{eqn:position} holds with $a_\infty$ and $b_\infty$ replaced by $\tilde{a}_\infty$ and $\tilde{b}_\infty$. The same applies to $c_\infty$ and $d_\infty$.

Since $n=3$, we know the two planes intersect in a line. By the discussion above, we may assume that $a_\infty=c_\infty$ by rotations and scaling. Together with $\abs{a_\infty}=\abs{b_\infty}=\abs{c_\infty}=\abs{d_\infty}$, the first equation in \eqref{eqn:position} implies that
\begin{equation*}
	\abs{a_{\infty}}^2 = \abs{c_\infty}^2= - b_\infty\cdot d_\infty = \abs{b}^2 =\abs{d}^2,
\end{equation*}
which forces $b_\infty=-d_\infty$. Therefore, $\mbox{Span}\set{a_\infty,b_\infty}=\mbox{Span}\set{c_\infty,d_\infty}$ and this contradicts our assumption that $a_\infty,b_\infty,c_\infty,d_\infty$ span $\Real^3$. This proved the first assertion in Corollary \ref{cor:minimal}. 

\begin{rem}
	We may also learn from the above proof that in case $q_\infty=0$, $(a_\infty,b_\infty)$ and $(c_\infty,d_\infty)$ span planes with different orientation. This will be proved again below.
\end{rem}

We now set up new coordinates on the plane so that $a_\infty=(1,0)$ and $b_\infty=(0,1)$. Setting $c_\infty=(x_c,y_c)$ and $d_\infty=(x_d,y_d)$, \eqref{eqn:q2} and \eqref{eqn:q3} becomes
\begin{eqnarray} \label{eqn:qq2}
	\abs{q_\infty}^2 &=& 2( x_c + y_d)  \\
	0&=& x_d -y_c . \label{eqn:qq3}
\end{eqnarray}
Recall that $c_\infty \perp d_\infty$, which gives
\begin{equation}\label{eqn:normal}
	x_c x_d + y_c y_d =0.
\end{equation}

{\bf Case 1.} ($x_d= y_c \ne 0$). \eqref{eqn:normal} gives that $x_c=-y_d$, which implies that $q_\infty=0$ by \eqref{eqn:qq2}. The orientation of the plane spanned by $(c_\infty,d_\infty)$ is given by the determinant
\begin{equation*}
	x_c y_d -x_d y_c = - (x_c^2 +y_c^2) <0,
\end{equation*}
which is opposite to the orientation given by $(a_\infty,b_\infty)$.

{\bf Case 2.} ($x_d=y_c=0$). By $\abs{c_\infty}=\abs{d_\infty}>0$, we have $\abs{x_c}=\abs{y_d}\ne 0$. When $q_\infty=0$, $x_c=-y_d$ and the orientations are opposite again; when $q_\infty\ne 0$, $x_c=y_d=\lambda>0$, the orientations are the same and $c_\infty=\lambda a_\infty$ and $d_\infty=\lambda b_\infty$, which implies that \eqref{eqn:vinf} parametrizes a catenoid.

\section{The index inequality}\label{sec:index}

In this section, we prove Theorem \ref{thm:main3}. We prove \eqref{eqn:ni} only and the proof of  \eqref{eqn:n} is the same. The proof studies the limit of eigenfunctions of $J_{u_i}$. As we know, the eigenfunctions depend on a choice of metric. Instead of a fixed metric, we construct a sequence of conformal metrics $g_i$ in the following subsection.

\subsection{A special sequence of conformal metrics} \label{sub:metrics}
Recall that both being a harmonic map and the $E$-index($E$-nullity) of a harmonic map are conformal invariant. Hence, we may assume that $g$ is flat in a neighborhood of $p$. Precisely, we assume that $x,y$ are isothermal coordinates near $p$ and $g=dx^2+dy^2$ in $B=\set{x^2+y^2\leq 1}$. We denote the polar coordinates associated to $(x,y)$ by $(r,\theta)$ and the radius $r$ ball centered at the origin by $B_r$.

The bubble map can be regarded as a harmonic map from $\Real^2$ instead of $S^2$. While the natural choice of metric on $\Real^2$ is either the flat metric or the spherical metric (the pull back of round metric by stereographic projection), it is convenient for us to use a modification (denoted by $g_b$) of the spherical metric on $\Real^2$ such that the geometry of a neighborhood of the infinity is that of a punctured flat disk (with the hole corresponding to the infinity). More precisely, we fix any smooth function $f:[0,\infty)\to \Real$ satisfying
	\begin{equation}\label{eqn:gbf}
		f(r)=\left\{
			\begin{array}[]{ll}
				\left( \frac{1}{1+r^2} \right)^2 & r\leq 1\\
				\frac{1}{r^4} & r>2
			\end{array}
			\right.
	\end{equation}
and define 
\begin{equation}\label{eqn:gb}
	g_b= f(r) (dr^2+ r^2 d\theta^2)
\end{equation}
where $(r,\theta)$ is the polar coordinates on $\Real^2$.
\begin{rem}
	In fact, it suffices for us to have $f=\frac{1}{r^4}$ for $r>2$.  The idea that $g_b$ is related to a round sphere is not necessary for the proof.
\end{rem}

Now let $u_i$ be the sequence in Theorem \ref{thm:main3}. In general, since $p$ is the only concentration point, there exist $x_i\to p$ and $\lambda_i\to 0$ such that 
\begin{equation*}
	u_i(x_i+ \lambda_i y)
\end{equation*}
converges locally smoothly to the bubble map $\omega$ (see (ii) in the introduction). For the sake of simplicity, we assume that $x_i=p$ for all $i$. The loss of generality is small and the structure of the proof remains the same.

The key to the proof of Theorem \ref{thm:main3} is a sequence of metrics $g_i$ associated to $u_i$ (or, $\lambda_i$, more precisely). For the definition of $g_i$, we recall the standard catenoid metric 
\begin{equation*}
	\left( \frac{e^t+e^{-t}}{2} \right)^2 (dt^2+d\theta^2),
\end{equation*}
or equivalently, by setting $r=e^t$,
\begin{equation*}
	\left( \frac{1+r^{-2}}{2} \right)^2 \left( dr^2+r^2 d\theta^2 \right).
\end{equation*}
Aligning the center of the catenoid with the center of the neck by replacing $r$ in the above formula by $\frac{r}{\sqrt{\lambda_i}}$ and multiplying the metric by $4\lambda_i$, we get
\begin{equation}\label{eqn:git}
	\tilde{g}_i:=\left( 1+\frac{\lambda_i}{r^2} \right)^2 \left( dr^2+r^2 d\theta^2 \right).
\end{equation}
Letting $\varphi:[0,\infty)\to [0,1]$ be a smooth function satisfying $\varphi(s)\equiv 0$ for $s\leq 1$ and $\varphi(s)\equiv 1$ for $s\geq 2$, we define
\begin{equation}\label{eqn:gi}
	g_i=\left\{
		\begin{array}[]{ll}
			g & \text{on} \quad M\setminus B_{1/2} \\
			\varphi(4r) g + (1-\varphi(4r)) \tilde{g}_i &\text{on} \quad B_{1/2}\setminus B_{1/4} \\
			\tilde{g}_i &\text{on} \quad B_{1/4} \setminus B_{ 4\lambda_i} \\
			\varphi(r/(2 \lambda_i)) \tilde{g}_i + (1-\varphi(r/(2 \lambda_i))) (L_i)^*g_b &\text{on} \quad B_{4 \lambda_i} \setminus B_{2 \lambda_i} \\
			(L_i)^* g_b &\text{on} \quad B_{2 \lambda_i},
		\end{array}
		\right.
\end{equation}
where $L_i:\Real^2 \to \Real^2$ maps $(x,y)$ to $\frac{1}{\lambda_i}(x,y)$. 

We now summarize some important properties of $g_i$, which will be useful in the proof of Theorem \ref{thm:main3}. We start with an easier observation.
\begin{lem}
	\label{lem:gi1} For any $\delta\in (0,\frac{1}{4})$, we have, when $i\to \infty$,

	(1) $g_i$ converges to $g$ on $M\setminus B_\delta$;

	(2) $(L_i^{-1})^* g_i$ converges to $g_b$ on $B_{1/\delta}$;

	(3) The volume of $B_\delta\setminus B_{\lambda_i/\delta}$ with respect to $g_i$ satisfies
	\begin{equation}
		\lim_{i\to \infty} \mbox{Vol}_{g_i}(B_\delta\setminus B_{\lambda_i/\delta}) \leq C \delta^2
		\label{eqn:vol}
	\end{equation}
	for some universal constant $C>0$.
\end{lem}
\begin{proof}
	By \eqref{eqn:gi}, it suffices to check (1) on $B_{1/2}\setminus B_\delta$. In fact, $\tilde{g}_i$ converges to $g$ on $B_{1/2}\setminus B_\delta$, which follows from \eqref{eqn:git} and the fact that $g= dr^2 + r^2 d\theta^2$ in $B_{1/2}\setminus B_\delta$ and that $\lim_{i\to\infty} \frac{\lambda_i}{r^2}=0$ for $r\in (\delta,\frac{1}{2})$.

	The proof of (2) follows from the claim that 
	\begin{equation*}
		\varphi(r/(2 \lambda_i)) (L_i^{-1})^* \tilde{g_i} + (1-\varphi(r/(2\lambda_i))) g_b
	\end{equation*}
	converges to $g_b$ on $B_{1/\delta}\setminus B_{2}$. To see this, we compute (by using \eqref{eqn:git})
	\begin{equation*}
		(L_i^{-1})^* \tilde{g}_i = \left( \lambda_i + \frac{1}{r^2} \right)^2 \left( dr^2+ r^2 d\theta^2 \right). 
	\end{equation*}
	The claim is proved by taking $i\to \infty$ in the above equation and comparing with \eqref{eqn:gb} and \eqref{eqn:gbf}.

	By \eqref{eqn:git}, we compute
	\begin{align*}
		\mbox{Vol}_{g_i} (B_\delta\setminus B_{\lambda_i/\delta}) &= 2\pi \int^\delta_{\lambda_i /\delta} \left( 1+ \frac{\lambda_i}{r^2} \right)^2 r dr\\
		&= 2\pi \lambda_i \int^{\delta/\sqrt{\lambda_i}}_{\sqrt{\lambda_i}/\delta} \left( 1+r^{-2} \right)^2 r dr \\
		&= 2\pi \lambda_i \int^{\log (\delta/\sqrt{\lambda_i})}_{\log (\sqrt{\lambda_i}/\delta)} \left( e^t + e^{-t}\right)^2  dt \\
		&\leq 16 \pi \lambda_i \int_0^{\log (\delta/\sqrt{\lambda_i})} e^{2t} dt \\
		&= 8\pi \lambda_i \left( \frac{\delta^2}{\lambda_i} - 1 \right). 
	\end{align*}
	Hence, (3) is proved with $C=8\pi$.
\end{proof}

The next property of $g_i$ follows from the obvious fact that it is a catenoid metric in the region $B_{1/4}\setminus B_{4\lambda_i}$. It is well known that the mean value inequality holds on minimal surfaces, see Corollary 1.16 in \cite{colding2011course} for example. Although we take an intrinsic point of view here, the mean value inequality carries over.

\begin{lem}
	\label{lem:gi2}
	For any positive number $C_1>0$, there is $C_2$ depending on $C_1$ but not $i$ such that if a nonnegative function $w$ satisfies
	\begin{equation*}
		\triangle_{g_i} w \geq - C_1 w,\qquad \text{on} \quad B_{1/8} \setminus B_{8 \lambda_i}
	\end{equation*}
	then for sufficiently large $i$,
	\begin{equation*}
		\sup_{B_{1/16}\setminus B_{16 \lambda_i}} w \leq C_2 \int_{B_{1/8}\setminus B_{8 \lambda_i}} w dV_{g_i}.
	\end{equation*}
\end{lem}
\begin{proof}
	Consider a scaling of the standard catenoid in $\Real^3$ (denoted by $\Sigma$) parametrized by
	\begin{equation*}
		\left\{
			\begin{array}[]{l}
				x_1=2\sqrt{\lambda_i} \frac{e^t+e^{-t}}{2}\cdot \cos\theta \\
				x_2=2\sqrt{\lambda_i} \frac{e^t+e^{-t}}{2}\cdot \sin\theta \\
				x_3=2\sqrt{\lambda_i} t.
			\end{array}
		\right.
	\end{equation*}

	By definition (see \eqref{eqn:gi}), $g_i$ is the induced metric and the equation \eqref{eqn:git} is the metric represented in coordinates $(r,\theta)$ where $r=\sqrt{\lambda_i} e^t$. By setting
	\begin{equation*}
		\Omega_1= \set{(x_1,x_2,x_3)\in \Sigma|\, t\in [\log (8\sqrt{\lambda_i}), -\log(8\sqrt{\lambda_i})]}
	\end{equation*}
	and
	\begin{equation*}
		\Omega_2= \set{(x_1,x_2,x_3)\in \Sigma|\, t\in [\log (16 \sqrt{\lambda_i}), -\log(16\sqrt{\lambda_i})]},
	\end{equation*}
	we notice that $\Omega_1$ and $\Omega_2$ corresponds to the domain $B_{1/8}\setminus B_{8\lambda_i}$ and $B_{1/16}\setminus B_{16\lambda_i}$ respectively. It is elementary to compute that
	\begin{gather*}
		\lim_{i\to \infty} \max_{\Omega_1} \abs{x_3} =0 \\
		\lim_{i\to \infty} \max_{\Omega_1} \sqrt{x_1^2+x_2^2}=\frac{1}{8} \\
		\lim_{i\to \infty} \max_{\Omega_2} \sqrt{x_1^2+x_2^2}=\frac{1}{16}.
	\end{gather*}
	As a consequence, for sufficiently large $i$ and any $y\in \Omega_2\subset \Real^3$, we have
	\begin{equation*}
		\hat{B}_{1/32}(y)\cap \partial \Omega_1 =\emptyset,
	\end{equation*}
	where $\hat{B}$ means the Euclidean ball in $\Real^3$. It then allows us to apply Corollary 1.16 of \cite{colding2011course} to show the existence of $C_2$ depending only on $C_1$ such that
	\begin{equation*}
		w(y)\leq C_2 \int_{\hat{B}_{1/32}(y)\cap \Omega_1} w dV_\Sigma.
	\end{equation*}
	Since $g_i$ is exactly the induced metric, our lemma follows.
\end{proof}

Here is another advantage of using $g_i$.
\begin{lem}
	\label{lem:gi3} There is a constant $C$ independent of $i$ such that $\abs{\nabla u_i}_{g_i}$ is uniformly bounded in $B_{1/8}\setminus B_{8\lambda_i}$.
\end{lem}
\begin{proof}
	The proof follows easily from Theorem \ref{thm:main1}. To see this, we switch to the cylinder coordinates $(t,\theta)$, where $t=\log r$. It follows from Theorem \ref{thm:main1} that there is some constant $C$ depending only on the sequence $u_i$ but not $i$ such that
\begin{equation*}
	\abs{\partial_t u_i}, \abs{\partial_\theta u_i} \leq C \left( e^t + \lambda_i e^{-t} \right)
\end{equation*}
for $t\in [\log (8 \lambda_i), \log (1/8)]$.
On the other hand, \eqref{eqn:gi} and \eqref{eqn:git} imply that for $t$ in the same range above, 
\begin{equation*}
	g_i= \left( e^t + \lambda_i e^{-t} \right)^2 (dt^2+d\theta^2).
\end{equation*}
\end{proof}

\subsection{The limit of eigenfunctions}
Now, let's prove Theorem \ref{thm:main3}. We give the proof of \eqref{eqn:ni} only and the proof of \eqref{eqn:n} is similar. By taking a subsequence, we may assume 
\begin{equation*}
	l:=\lim_{i\to\infty} \mbox{NI}(u_i).
\end{equation*}
Recall that the $E$-index form $\mathcal H_{u_i}$ is conformal invariant and the definition of the Jacobian operator $J_{u_i}$ does depend on a choice of metric on $M$, for which we use $g_i$ constructed above. With respect to $g_i$, there are nonpositive eigenvalues $\beta_{i,1},\cdots,\beta_{i,l}$ and eigenfunctions $v_{i,1},\cdots,v_{i,l}$ such that
\begin{equation}\label{eqn:vik}
	J_{u_i}(v_{i,k})=\beta_{i,k} v_{i,k}
\end{equation}
and these eigenfunctions are normalized so that
\begin{equation}\label{eqn:normalize}
	\int_M \langle v_{i,k}, v_{i,k'} \rangle dV_{g_i}= \delta_{k,k'}.
\end{equation}

\begin{rem}
	Here $v_{i,k}$ are sections of the pullback bundle $u_i^*TN$. Recall that $N$ is embedded into $\Real^p$ so that $u_i$ are $\Real^p$-valued functions. We hence regard $v_{i,k}$ as $\Real^p$-valued functions that are perpendicular to the tangent space of $N$ at $u_i$. We refer to Section \ref{sub:equation} for details.
\end{rem}

For the proof of Theorem \ref{thm:main3}, we study the limit of $v_{i,k}$'s. This requires several apriori bounds. The first one is a lower bound of the eigenvalues.

\begin{lem}
	\label{lem:lowerbound} There is a constant $C>0$ depending on the sequence $u_i$ but not $i$ such that
	\begin{equation*}
		\beta_{i,k}\geq -C
	\end{equation*}
	for $k=1,\cdots,l$ and all $i$.
\end{lem}
\begin{proof}
By the definition of eigenvalues, we have
\begin{equation*}
	\int_M \langle  J_{u_i}(v_{i,k}), v_{i,k} \rangle dV_{g_i}= \beta_{i,k} \int_M \abs{v_{i,k}}^2 dV_{g_i}.
\end{equation*}
By \eqref{eqn:J},
\begin{equation*}
	- \int_M \langle \nabla_{u_i}^2 v_{i,k} + R(\nabla u_i, v_{i,k})\nabla u_i,v_{i,k} \rangle dV_{g_i}= \beta_{i,k} \int_M \abs{v_{i,k}}^2 dV_{g_i}.
\end{equation*}
Using integration by parts and the fact that curvature of $N$ is bounded, we obtain
\begin{equation}\label{eqn:reuse}
	\int_M  \abs{\nabla_{u_i} v_{i,k}}^2 - C(N)\abs{\nabla u_i}_{g_i}^2 \abs{v_{i,k}}^2 dV_{g_i} \leq \beta_{i,k} \int_M \abs{v_{i,k}}^2 dV_{g_i}.
\end{equation}
The proof of the lemma then follows Lemma \ref{lem:gi3}.
\end{proof}

We also have apriori estimates for $v_{i,k}$.
\begin{lem}
	\label{lem:vik}
	There is a constant $C$ independent of $i$ such that
	\begin{equation}\label{eqn:l2}
		\int_M \abs{\nabla_{u_i}v_{i,k}}^2_{g_i} dV_{g_i}\leq C
	\end{equation}
	and
	\begin{equation}\label{eqn:Linfinity}
		\sup_{B_{1/16}\setminus B_{16\lambda_i}} \abs{v_{i,k}} \leq C.
	\end{equation}
\end{lem}
\begin{rem}
	Indeed, we have $\sup_M \abs{v_{i,k}}$ bounded, which will be clear in a minute.
\end{rem}
\begin{proof}
	The proof of \eqref{eqn:l2} follows from \eqref{eqn:reuse}, Lemma \ref{lem:gi3} and the fact that $\beta_{i,k}\leq 0$.

	Using \eqref{eqn:vik}, we compute
	\begin{eqnarray*}
		\triangle_{g_i} \abs{v_{i,k}}^2 &=& 2 \mbox{div}_{g_i} \langle \nabla_{u_i} v_{i,k},v_{i,k}\rangle \\
		&=& 2 \langle \mbox{Tr}_{g_i} \nabla^2_{u_i} v_{i,k},v_{i,k} \rangle + 2 \abs{\nabla_{u_i} v_{i,k}}^2_{g_i} \\
		&\geq& - 2 \beta_{i,k} \abs{v_{i,k}}^2 - 2\langle \mbox{Tr}_{g_i} R(du_i,v_{i,k})du_i, v_{i,k} \rangle \\
		&\geq & - C \abs{v_{i,k}}^2
	\end{eqnarray*}
	for some $C>0$ independent of $i$.
	Here, we have used Lemma \ref{lem:gi3} and the fact that $\beta_{i,k}\leq 0$. We can now apply Lemma \ref{lem:gi2} to conclude the proof of \eqref{eqn:Linfinity}.
\end{proof}

Lemma \ref{lem:lowerbound} and Lemma \ref{lem:vik} allow us to consider the limit of the pair $(\beta_{i,k},v_{i,k})$ for $k=1,\cdots,l$. As a first step, by taking subsequence, we assume that
\begin{equation*}
	\lim_{i\to\infty} \beta_{i,k}=\beta_k,\qquad k=1,\cdots,l.
\end{equation*} 
Keep in mind that the background metric $(M,g_i)$ also changes and fix some arbitrary small $\delta>0$.

On $M\setminus B_\delta$, $g_i$, $u_i$ and $\beta_{i,k}$ converge to $g$, $u_\infty$ and $\beta_k$ respectively. Hence the coefficients of the following elliptic equation (see Section \ref{sub:equation})
\begin{equation*}
	J_{u_i}(v_{i,k})=\beta_{i,k} v_{i,k}
\end{equation*}
are uniformly bounded by a constant depending on $\delta$. Since $\int_M \abs{v_{i,k}}^2 dV_{g_i}=1$, up to a subsequence, we assume that $v_{i,k}$ converges to $v_k$, which is a vector field along $u_\infty$. By the arbitrariness of $\delta$, $v_k$ is defined on $M\setminus \set{p}$ and is a solution to the equation
\begin{equation*}
	J_{u_\infty} (v_k)= \beta_k v_k.
\end{equation*}
Recall that both $g$ and $u_\infty$ are in fact smooth at $p$. Lemma \ref{lem:vik} implies that $v_k$ is a bounded weak solution to the above linear equation. The usual $L^p$ theory and the bootstrapping argument show that $v_k$ is smooth. Thus, we have found a smooth eigenfunction $v_k$ of $J_{u_\infty}$ corresponding to the eigenvalue $\beta_k$.


On the other hand, we may consider the convergence of $v_{i,k}|_{B_{\lambda_i/\delta}}$. More precisely, we study the scaling $v_{i,k}\circ L_i^{-1}$ that is defined on $B_{1/\delta}$. Notice that $(L_i^{-1})^*g_i$ converges to $g_b$ on $B_{1/\delta}$ by Lemma \ref{lem:gi1} and $u_i\circ L_i^{-1}$ converges to $\omega$ by the definition of the bubble. Due to the $L^2$ bound
\begin{equation*}
	\int_{B_{\lambda_i/\delta}} \abs{v_{i,k}}^2 dV_{g_i}\leq 1
\end{equation*}
and a similar argument as above, we know that there is a limit vector field along $\omega$, denoted by $\tilde{v}_k$ defined on $\Real^2$, which satisfies
\begin{equation*}
	J_\omega (\tilde{v}_k) = \beta_k \tilde{v}_k.
\end{equation*}
Notice that in the above equation, $g_b$ is involved implicitly in the definition of $J_\omega$. Moreover, $g_b$ can be regarded as a complete metric on $S^2$ such that $(\Real^2,g_b)$ is just $(S^2,g_b)$ with a point removed. Both $g_b$ and $\omega$ are smooth at this point. Hence, Lemma \ref{lem:vik} again implies that this is a removable singularity of $\tilde{v}_k$ and we obtain a smooth eigenfunction $\tilde{v}_k$ of $J_{\omega}$ corresponding to the eigenvalue $\beta_k$.

To finish the proof, we study the limit of \eqref{eqn:normalize}.  Using (3) of Lemma \ref{lem:gi1} and \eqref{eqn:Linfinity} of Lemma \ref{lem:vik}, we get (by sending $i\to \infty$ and then $\delta\to 0$)
\begin{equation}
	\int_{M} \langle v_k,v_{k'} \rangle dV_g + \int_{S^2} \langle \tilde{v}_k, \tilde{v}_{k'} \rangle dV_{g_b} = \delta_{k,k'}	
	\label{eqn:final}
\end{equation}
for $k,k'=1,\cdots,l$. While $\set{v_k}$ ($\set{\tilde{v}_{k}}$) may not be linearly independent, the ranks of the semi-positive definite symmetric matrices
\begin{equation*}
	\mathcal M_1:=\int_{M} \langle v_k,v_{k'} \rangle dV_g,\qquad \mathcal M_2:=\int_M \langle \tilde{v}_k, \tilde{v}_{k'} \rangle dV_{g_b}	
\end{equation*}
are the dimensions of linear spaces spanned by $v_k$ and $\tilde{v}_{k}$ respectively, which is no larger than $\text{NI}(u_\infty)$ and $\text{NI}(\omega)$. The proof of Theorem \ref{thm:main3} now follows from the linear algebra theorem that $\text{rank}(\mathcal M_1) + \text{rank}(\mathcal M_2)\geq l$ because of \eqref{eqn:final}.

\bibliographystyle{alpha}
\bibliography{foo}
\end{document}